\documentclass[10pt]{amsart}
\usepackage{amsfonts}
\usepackage{array}
\usepackage{tabularx}
\usepackage{arydshln}
\usepackage{amsmath}

\usepackage{amssymb}
\usepackage{physics}
\usepackage{graphicx}%
\usepackage{caption}
\usepackage{subcaption}
\usepackage{multirow}
\usepackage{footmisc}
\providecommand{\U}[1]{\protect\rule{.1in}{.1in}}
\newtheorem{theorem}{Theorem}

\newtheorem{corollary}[theorem]{Corollary}

\newtheorem{definition}[theorem]{Definition}

\newtheorem{lemma}[theorem]{Lemma}

\newtheorem{proposition}[theorem]{Proposition}

\usepackage{xcolor}

\usepackage{latexsym}
\usepackage{amsmath}
\usepackage{amssymb}
\usepackage{mathrsfs}
\usepackage{graphicx}
\usepackage{color}
\usepackage{pgfpages}
\usepackage{ifthen}
\usepackage{leftidx,tensor}
\usepackage[T1]{fontenc}
\usepackage[latin1]{inputenc}
\usepackage{mathtools}
\usepackage{comment}
\usepackage{dsfont}
\usepackage[nocompress]{cite}

\usepackage[shortlabels]{enumitem}
\usepackage{aliascnt}
\usepackage[bookmarks=true,pdfstartview=FitH, pdfborder={0 0 0}, colorlinks=true,citecolor=red, linkcolor=blue]{hyperref}
\usepackage{nicefrac}

\newaliascnt{cor}{thm}
\newaliascnt{prop}{thm}
\newaliascnt{lem}{thm}
\aliascntresetthe{cor}
\aliascntresetthe{prop}
\aliascntresetthe{lem} 
\newaliascnt{defn}{thm}
\newaliascnt{asu}{thm}
\newaliascnt{con}{thm}
\aliascntresetthe{defn}
\aliascntresetthe{asu}
\aliascntresetthe{con}
\newcounter{stp}
\newcounter{stpi}
\newcounter{stpci}
\newcounter{stpiii}

\newaliascnt{rem}{thm}
\newaliascnt{exa}{thm}
\newaliascnt{masu}{thm}
\newaliascnt{nota}{thm}
\newaliascnt{sett}{thm}
\aliascntresetthe{rem}
\aliascntresetthe{exa}
\aliascntresetthe{masu}
\aliascntresetthe{nota}
\aliascntresetthe{sett}
\numberwithin{equation}{section}

\setlist[enumerate]{font = \normalfont}

\newcommand {\R}	{\mathbb{R}}

\newcommand {\E}	{\mathbb{E}}

\renewcommand{\d}{\, \mathrm{d}}

\DeclareMathOperator{\divH}{div_{\H}}

\renewcommand{\H}{\mathrm{H}}

\newcommand{\per}{\mathrm{per}}

\newcommand{\sigmabar}{\bar{\sigma}}

	\newcommand{\dk}[1]{\partial_{#1}}
	\newcommand{\dt}{\dk{t}} 
	\newcommand{\dz}{\dk{z}} 
	
	\newcommand{\eps}{\varepsilon}
	\renewcommand{\phi}{\varphi}
	
	\renewcommand{\bar}[1]{\overline{#1}}
	\newcommand{\vbar}{\bar{v}}

	\renewcommand{\div}{\mathrm{div} \, }
	\newcommand{\nablaH}{\nabla_{\H}}
	\newcommand{\DeltaH}{\Delta_{\H}}

	\newcommand{\rC}{\mathrm{C}}
	\newcommand{\rL}{\mathrm{L}}
	\newcommand{\rW}{\mathrm{W}}
	\newcommand{\rH}{\H}
	\newcommand{\rB}{\mathrm{B}}

	\newcommand{\rLq}{\rL^q}

	\newcommand{\rX}{\mathrm X}

\title[Time-Periodic Solutions to an Energy Balance Model for Arbitrarily large forces]{Time-periodic solutions to an energy balance  model coupled with an active fluid under arbitrary large forces}

\author{Gianmarco Del Sarto}
\address{Technische Universit\"{a}t Darmstadt\\
Fachbereich Mathematik\\
	Schlossgartenstr.\ 7\\
	64289 Darmstadt\\
	Germany}
\email{delsarto@mathematik.tu-darmstadt.de}

\author{Matthias Hieber}
\address{Technische Universit\"at Darmstadt\\
	Fachbereich Mathematik\\
	Schlossgartenstr.\ 7\\
	64289 Darmstadt\\
	Germany}
\email{hieber@mathematik.tu-darmstadt.de}

\author{Filippo Palma}
\address{Universit\`a degli Studi della Campania L. Vanvitelli\\
	Dipartimento di Matematica e Fisica\\
	Via Vivaldi 43\\
	81100 Caserta\\
	Italy}
\email{filippo.palma@unicampania.it}

\author{Tarek Z\"{o}chling}
\address{Technische Universit\"at Darmstadt\\
	Fachbereich Mathematik\\
	Schlossgartenstr.\ 7\\
	64289 Darmstadt\\
	Germany}
\email{zoechling@mathematik.tu-darmstadt.de}

\begin{document}


\begin{abstract}
This article concerns time-periodic solutions to a two-dimensional Sellers-type energy balance model coupled to the three-dimensional primitive equations via a dynamic boundary condition.  It is
shown that the underlying equations admit at least one strong time-periodic solution, provided the forcing term is time-periodic. The forcing term does not need to satisfy a smallness condition and is allowed to
be arbitrarily large.
\end{abstract}

\maketitle

\section{Introduction  }\label{Xsec1-1}\label{sec: intro}

The theory of periodic solutions to ordinary and partial differential equations has a very long and rich history. Several  approaches to various  solution concepts (weak, strong, etc.)
have been developed. When considering equations arising from the theory of incompressible, viscous fluid flows, first results on the existence of periodic solutions for the Navier-Stokes equations go
back to Serrin \cite{bib0037} and Prodi \cite{bib0035}. Further results in this direction are due e.g., to Maremonti \cite{bib0033} and Kozono and Nakao \cite{bib0029},  Galdi and Sohr \cite{bib0020}, Yamazaki \cite{bib0041} and
Galdi and Silvestre  \cite{bib0019}. A general approach to periodic solutions for incompressible fluid flows was developed by Geissert, Hieber and Nguyen \cite{bib0021}. Their method was based on the interpolation
and smoothing properties of the linearized problem. We also refer here to the survey article by Galdi and Kyed \cite{bib0018} and the references therein.

A characterization of strong periodic solutions for linear evolution equations  within the $\rL^p$-setting for $1<p<\infty$ was obtained by Arendt and Bu \cite{bib0002}. For an extension of this result to the
semilinear and quasilinear setting, we refer to the work of Hieber and Stinner \cite{bib0028}; for an approach based on the Da Prato-Grisvard theorem, we refer to \cite{bib0004}.
Let us emphasize that all these approaches have the drawback that the external force needs to satisfy a smallness condition with respect to a certain norm.

In this article we investigate a system coupling a Sellers-type energy balance model with the primitive equations. Note that the coupling is described via a dynamical boundary
condition for the temperature at the interface. We show that the underlying coupled system of equations
admits at least one strong time-periodic solution for time periodic forces $f \in \rL^2(0,\tau;\rL^2(\Omega))$ {\em without assuming any smallness condition on $f$}.
This is fairly surprising since all the approaches cited above need certain smallness conditions on the outer force $f$. Let us note that a result  of this type is not known e.g., for the
classical Navier-Stokes equations or the Keller-Segel system.

Let us now describe an energy balance model (EBM) coupled to the primitive equations in some detail. Energy balance models  belong to a class of simplified climate models and represent the Earth's
climate by balancing the incoming solar radiation with the outgoing terrestrial radiation \cite{bib0005,bib0034,bib0036}. Despite their relative simplicity compared to full-scale climate models,
EBMs capture the essential physics that govern Earth's climate and provide valuable insights into climate dynamics and feedback mechanisms. They thus form  a useful tool in both theoretical studies
and practical applications. For instance, questions regarding the well-posedness, uniqueness, and long-time behaviour of this model, when not coupled to a fluid, have been investigated \cite{Diaz1997,Diaz2022,Hetzer_01}, as well as controllability and inverse problems related to the recovery of model parameters \cite{bib0014,bib0006}, and applications to understanding climate change \cite{bib0008,bib0009,bib0023}. Diaz and Tello investigated in \cite{bib0013} an energy balance model for a {\em given} fluid subject to a dynamic
boundary condition and proved the existence of a weak solution to this system.

Our coupled model consists in coupling the energy balance model  with the fundamental equation for
geophysical flows, the primitive equations. The latter have been introduced and investigated by Lions, Temam and Wang in a series of articles, see \cite{bib0030,bib0031,bib0032}. These equations describe the
large-scale motion of the ocean or the atmosphere by coupling a horizontal momentum equations with a vertical hydrostatic balance and an incompressibility condition. They  can be derived rigorously
from the Navier-Stokes equations under the assumption of hydrostatic balance, see \cite{bib0015} for details. At this stage it is useful to highlight the built-in time-scale separation of the coupled model: the primitive-equations (PE) dynamics evolve on a fast time scale \(\tau_{\mathrm{fast}}\) (synoptic to seasonal), whereas the energy-balance model (EBM) evolves on a slow time scale \(\tau_{\mathrm{slow}}\) (from days-months up to years-decades). To encode this separation, we introduce the non-dimensional parameter
\[
\varepsilon \coloneqq \frac{\tau_{\mathrm{fast}}}{\tau_{\mathrm{slow}}}, \qquad 0<\varepsilon\ll 1,
\]
and rescale the PE time variable by \(\varepsilon\). In the non-dimensional PE system this places a factor \(\varepsilon\) in front of the time derivative, making the fast atmospheric dynamics explicit relative to the slower EBM evolution. For more details, we refer to \cite{bib0010}.

Given $\tau > 0$, the energy balance model coupled to the primitive equations reads as
\begin{equation}
\left\{
\begin{aligned}
\eps \, \partial_t v^\eps + u^\eps \cdot \nabla v^\eps - \Delta v^\eps+ \nablaH p^\eps &= f^\eps_1,  &\quad \text{in } (0,\tau) \times  \Omega  ,\\[1mm]
\partial_z p^\eps &= -T,  &\quad \text{in } (0,\tau) \times \Omega,\\[1mm]
\operatorname{div} u^\eps &= 0, &\quad \text{in } (0,\tau) \times \Omega,\\[1mm]
\partial_t T + u^\eps \cdot \nabla T - \Delta T &= f_2,  &\quad \text{in } (0,\tau) \times \Omega ,\\[1mm]
T|_{\Gamma_u} &= \rho,  &\quad \text{in } (0,\tau) \times G ,\\[1mm]
\partial_t \rho + \bar{v}^\eps \cdot \nablaH \rho - \DeltaH \rho + (\partial_z T)|_{\Gamma_u} &= R(x,\rho) + f_3,  &\quad \text{in } (0,\tau) \times G .
\end{aligned}
\right.
\label{eq: primitive + EBM}
\end{equation}
Here, \(\Omega = G \times (0,1)\), with \(G \subset \mathbb{R}^2\) representing  the horizontal domain, $\Omega$ denotes the spatial domain for the ocean  and \(\Gamma_u = G \times \{1\}\)
its upper (surface) boundary. Moreover, $u = (v, w):  (0,\tau) \times \Omega  \rightarrow \mathbb{R}^3$ denotes the velocity field of the ocean governed by the primitive equations,
$p : (0,\tau) \times \Omega \rightarrow \mathbb{R}$ is the pressure and  $T :  (0,\tau) \times \Omega  \rightarrow \mathbb{R}$ the temperature of the ocean. We denote the horizontal gradient, divergence,
and Laplacian by $\nablaH := (\partial_x,\partial_y)^T,$ $\divH := \nablaH \cdot,$ and $\DeltaH := \partial_x^2 + \partial_y^2$ respectively.  The system is supplemented with appropriate boundary conditions described in
\autoref{sec: prelim + main} in detail.  The surface temperature  of the ocean is given by $\rho = T|_{\Gamma_u} : (0,\tau) \times G  \rightarrow \mathbb{R}$, while $\vbar (x,y) = \int_0^1 v(x,y,\xi) \d \xi $ is the vertical average of the horizontal velocity.

The functions  $f_1, f_2$ and $f_3$ represent the periodic external forces acting on the velocity and the temperature, respectively. The consideration of periodic forcing is motivated by the
various cyclic phenomena affecting the Earth's climate system, ranging from daily and seasonal solar radiation to decadal solar cycles and Milankovitch orbital cycles, see e.g.,
\cite{bib0034}.

Note that surface temperature $\rho$ in the last line of \eqref{eq: primitive + EBM} is given by a dynamic boundary condition. Hence, the above system cannot be treated by the methods
developed in \cite{bib0007}.

Taking a more abstract perspective, it seems that no results are known concerning strong, periodic solutions for systems where the  forces are acting  by means of  dynamic boundary conditions.
Let us note that Badii and Diaz \cite{bib0003} established the existence of a weak periodic solution to an energy balance model not being coupled to a fluid.
Hence, our result on the   energy balance model coupled to the primitive equations seems to be the first result for systems subject to  dynamic boundary conditions allowing for  strong, periodic
solutions. Furthermore, as written above, it is surprising that our result holds true for arbitrary large forces $f$.

Some words about our strategy to prove the existence of time-periodic solutions to the coupled energy balance model are in order. Our approach is inspired by
the one developed by Galdi, Hieber, and Kashiwabara \cite{bib0017}, who established  the existence of strong time-periodic solutions to the primitive equations.
We start by showing  the existence of a weak \(T^*\)-periodic
solution to \eqref{eq: primitive + EBM} for arbitrary \(T^*\)-periodic forcing terms by using Galerkin's method and Brouwer's fixed point theorem. In a second step,
we extend the recent global strong well-posedness result to \eqref{eq: primitive + EBM} for initial data in $H^1(\Omega)$ given in  \cite{bib0010}.
Observe that obtaining this global well-posedness result is delicate due to the dynamic boundary condition for $\rho$ at the boundary.
Finally,  we establish a weak-strong uniqueness result to conclude the main result.

This paper is organized as follows. In \autoref{sec: prelim + main}, we define the notion of a weak and strong time-periodic solution and state our main result.
In \autoref{sec: global}, the existence and uniqueness of a strong solution to the initial boundary value problem \eqref{eq: primitive + EBM} for initial data in $H^1(\Omega)$ is proven.
Furthermore, we prove in  \autoref{sec:WS} the existence of a weak \(T^*\)-periodic solution to \eqref{eq: primitive + EBM} for arbitrary large \(T^*\)-periodic forcing terms using
Galerkin's method and Brouwer's fixed point theorem. Finally, a  weak-strong uniqueness result is established in \autoref{sec:w-s-u}, which guarantees the validity of our main theorem.

\section{Preliminaries and Main Results}\label{Xsec2-2}\label{sec: prelim + main}
Let $\tau>0$ be given. Consider a cylindrical domain $\Omega = G \times (0,1)$, with $G=(0,1)^2$, and denote by $\Gamma $ its boundary, i.e.,
\[
\partial \Omega = \Gamma = \Gamma_u \cup \Gamma_b \cup \Gamma_l ,
\]
where $\Gamma_u, \Gamma_b, \Gamma_l$ denote the respective upper, bottom and lateral parts of the boundary. They  are given by
\[
\Gamma_ u = (0,1)^2 \times \{ 1 \}, \quad \Gamma_b = (0,1)^2 \times \{0\}, \quad \Gamma_l = \partial G \times [0,1].
\]
For simplicity of the notation we denote in the following the rescaled oceanic velocity $u^\eps$ and the rescaled pressure $p^\eps$ by $u$ and $p$ respectively. Moreover, as it does not affect the analysis, we set the scaling constant $\eps \equiv 1$. We aim to study the following set of equations
\begin{equation*}
\left\{
\begin{aligned}
\partial_t v + u \cdot \nabla  v- \Delta v + \nablaH p &= f_1,  &\quad \text{ in } & (0,\tau) \times \Omega,\\
\dz p &= -T,  &\quad \text{ in } & (0,\tau) \times \Omega,\\
\div u &=0, &\quad \text{ in } & (0,\tau) \times \Omega ,\\
\partial_t T + u \cdot \nabla T -\Delta T&= f_2, \quad &\text{ in }& (0,\tau) \times \Omega, \\
 T|_{\Gamma_u} &= \rho, \quad &\text{ in }& (0,\tau) \times G ,\\
\partial_t \rho + \vbar \cdot \nablaH  \rho- \DeltaH \rho + (\dz  T)|_{ \Gamma_u}  &= R(x,\rho) +f_3, \quad &\text{ in }& (0,\tau) \times G ,\\
\end{aligned}
\right.
\end{equation*}
where $u = (v, w) : (0,\tau) \times \Omega  \rightarrow \R^3$ denotes the ocean's velocity, $p : (0,\tau) \times \Omega  \rightarrow \R$ its pressure, $T :(0,\tau) \times \Omega \to \R$ the ocean's temperature, $\rho =T|_{\Gamma_u}: (0,\tau) \times G \to \R$ the temperature evaluated at the surface and $\vbar (x,y) = \int_0^1 v(x,y,\xi) \d \xi $ the vertical average of the horizontal velocity. Moreover, $f_1 \colon (0,\tau) \times \Omega  \to \R^2$, $f_2 \colon  (0,\tau) \times \Omega \to \R$ and $f_3 \colon (0,\tau) \times G  \to \R$ denote external forces for the velocity and temperature respectively.

The above system is supplemented with the boundary conditions
\begin{equation}\label{eq:bc}
\begin{aligned}
  (\dz v)|_{\Gamma_u \cup \Gamma_b} = w|_{\Gamma_u \cup \Gamma_b} = 0, \quad (\dz T)|_{\Gamma_b}=0 \ \text{ and } \ u,p,T,\rho \quad \text{ are periodic on }\ \partial G \times [0,1],
  \end{aligned}
\end{equation}
as well as the initial data
\begin{equation}\label{eq: initial data}
     v(0) = v_0, \quad T(0)= T_0.
\end{equation}
The surface temperature $\rho$ is governed by the two-dimensional energy balance model
\begin{equation}\label{eq:EBM}
\partial_t \rho +\vbar \cdot \nablaH  \rho - \DeltaH \rho  + (\partial_z T)\vert_{\Gamma_u} = R(x,\rho) + f_3, \quad \text{in }  (0,\tau) \times G.
\end{equation}
In the above, all physical constants (such as the heat capacity and diffusion constants) have been normalized to one. In \eqref{eq:EBM}, the term \(\DeltaH \rho\) models the horizontal diffusion of temperature, \(\vbar\cdot \nablaH \rho\) captures the horizontal advection at the surface in a simplified way, and \((\partial_z T)\vert_{\Gamma_u}\) represents the vertical heat flux from the ocean interior to the surface. Note that the transport on the boundary is simplified by considering $\vbar \cdot \nablaH \rho$.
The reaction term is given by
\begin{equation}\label{eq: radiation}
R(x, \rho) = Q(x)\,\beta(\rho) - |\rho|^3 \rho, \quad x \in \Gamma_u,\; \rho \in \mathbb{R},
\end{equation}
where the solar radiation \(Q \in C^1_b(G)\) is positive, and the outgoing radiation is modeled via a Stefan-Boltzmann law. The co-albedo function is parametrized as
\begin{equation}\label{eq: coalbedo}
\beta(\rho) = \beta_1 + (\beta_2 - \beta_1)\, \frac{1+\tanh(\rho-\rho_{\mathrm{ref}})}{2},
\end{equation}
with \(0<\beta_1<\beta_2\) corresponding to the co-albedo values for ice-covered and ice-free conditions, respectively. The parameter $\rho_{\mathrm{ref}}$ denotes a critical reference temperature, usually chosen as $\rho_{\mathrm{ref}} =263\, \text{K}$,
at which ice turns white, see \cite{Diaz2022}. We point out that our EBM is of Sellers-type, meaning that the co-albedo is Lipschitz continuous. Another important class of EBMs is the Budyko-type, characterised by a discontinuous co-albedo. For the latter, it is known that, even for the EBM alone, uniqueness of solutions may fail; see, for instance, \cite{Diaz1997} for a discussion of the non-uniqueness of weak solutions. We refer
to \cite{bib0010} and the references therein for more information  on EBMs.

The equation of state $\dz p = -T$ yields
\begin{equation} \label{eq:press}
    p(t,x,y,z) = p_s(t,x,y) -  \int_0^z T(\cdot, \xi) \d \xi =:p_s(t,x,y) - \theta(t,x,y,z) ,
\end{equation}
where $p_s$ denotes the surface pressure $p_s(t,x,y)=p(t,x,y,z=0)$.
Moreover, the vertical velocity $w=w(v)$ is determined by the divergence free condition, i.e.,
\begin{equation} \label{eq:w}
    w(t,x,y,z) = -\int_0^z \divH v(t,x,y,\xi) \d \xi.
\end{equation}
Using this as well as the boundary conditions for $w$ we rewrite the system \eqref{eq: primitive + EBM} as
\begin{equation}
\left\{
\begin{aligned}
\partial_t v + v \cdot \nablaH  v+ w(v)\cdot \dz v - \Delta v + \nablaH p_s &= f_1 +\nablaH \theta,  &\quad \text{ in } & (0,\tau) \times \Omega ,\\
\divH \vbar &=0, &\quad \text{ in } &(0,\tau) \times G ,\\
\partial_t T + u \cdot \nabla T -\Delta T&= f_2, \quad &\text{ in }& (0,\tau) \times \Omega , \\
 T|_{\Gamma_u} &= \rho, \quad &\text{ in }& (0,\tau) \times G,\\
\partial_t \rho + \vbar \cdot \nablaH  \rho- \DeltaH \rho + (\dz  T)|_{ \Gamma_u}  &= R(x,\rho)+f_3, \quad &\text{ in }& (0,\tau) \times G,\\
\end{aligned}
\right.
\label{eq: primitive + EBM simplified}
\end{equation}
supplemented by the boundary conditions \eqref{eq:bc} as well as the initial data \eqref{eq: initial data}.

\noindent
At this point, some words about our approach for solving \eqref{eq: primitive + EBM simplified} are in order.
Let $m \in \mathbb{N}$ be a positive number and $q \in (1,\infty)$.
As usual, we denote by  $\rL^q(\Omega)$ and $\rH^{m,q}(\Omega) = \rW^{m,q}(\Omega)$ respectively the Lebesgue and Sobolev spaces. For more information on function spaces we refer e.g. to \cite{bib0001} and \cite{bib0040}. In the following, we need terminology to describe periodic boundary conditions on $\Gamma_l = \partial G \times [0,1]$ as well as on $\partial G$. Given a positive integer $ m \in \mathbb{N}$ and $q \in (1,\infty)$, we define the spaces
\begin{equation*}
    \rH^{m,q}_\per (\Omega) := \bar{\rC^\infty_\per (\bar{\Omega})}^{\| \cdot \|_{\rH^{m,q}(\Omega)}} \ \text{ and } \ \rH^{m,q}_\per (G) := \bar{\rC^\infty_\per (\bar{G})}^{\| \cdot \|_{\rH^{m,q}(G)}},
\end{equation*}
where horizontal periodicity is modeled by the function spaces $\rC^\infty_\per(\bar{\Omega})$ and $\rC^\infty_\per(\bar{G})$ defined in \cite{bib0027}. Of course, we interpret $\rH^{0,q}_\per$ as $\rLq$. For more information on spaces equipped with periodic boundary conditions, see also \cite{bib0024,bib0025,bib0027}. Next, we define \emph{hydrostatically solenoidal vector fields} by
\begin{equation*}
    \rL^q_{\sigmabar}(\Omega) = \overline{\{v \in \rC^\infty(\overline{\Omega};\R^2) \colon \divH \vbar = 0 \}}^{\| \cdot \|_{\rL^q(\Omega)}}.
\end{equation*}
Its role for the primitive equations parallels the one of the solenoidal vector fields for the Navier-Stokes equations. Next, we introduce the function spaces $\rX_0$, $\rX_1$ and $\rX_w$ by
\begin{equation*}
    \begin{aligned}
        \rX_0 &:= \rL^2_{\sigmabar}(\Omega;\R^2) \times \rL^2(\Omega) \times \rL^2(G),  \\
        \rX_1 & :=\{ v \in \rH^2_\per (\Omega;\R^2) \cap \rL^2_{\sigmabar} (\Omega;\R^2) \colon (\dz v)_{\Gamma_u \cup \Gamma_b} = 0 \} \times  \{ T\in \rH^2_\per(\Omega) : \ (\dz T)|_{\Gamma_b}=0 \}\times \rH^2_{\per}(G), \text{ and}  \\  \rX_w &:= (\rX_0,\rX_1)_{\frac{1}{2},2} =\rH^1_\per(\Omega;\R^2) \cap \rL^2_{\sigmabar}(\Omega;\R^2) \times \rH^1_\per(\Omega) \times \rH^1_\per(G),
    \end{aligned}
\end{equation*}
where $(\cdot, \cdot)_{\theta,p}$ denotes the real interpolation functor for $\theta \in (0,1)$ and $p \in (1,\infty)$. Then given $0 < \tau \leq \infty$ we define the data space $\E_{0,\tau}$ and the solution spaces $\E_{1,\tau}$ as well as $\E_{w,\tau}$ by
\begin{equation*}
    \begin{aligned}
        \E_{0,\tau} := \rL^2(0,\tau;\rX_0), \enspace \E_{1,\tau} := \rH^1(0,\tau;\rX_0) \cap \rL^2(0,\tau; \rX_1), \ \text{ and }\ \E_{w,\tau} :=  \rC_w([0,\tau];\rX_0) \cap \rL^2(0,\tau; \rX_w) .
    \end{aligned}
\end{equation*}

\noindent
We are now in position to define the solution concepts being used throughout this article. We only premise that a function $f\in L^2(0,\tau;L^2(\Omega))$, all $\tau>0$, is called $T^*$-periodic, $T^*>0$, if $f(t,x)=f(t+T^*,x)$, for all $(t,x)\in [0,+\infty)\times \Omega$.

\begin{definition}\label{def: weak sol}{\rm
  Let $T^*>0$ and $(f_1,f_2,f_3) \in \E_{0,\tau}$ for all $\tau>0$. Let $(f_1,f_2,f_3)$ be $T^*$-periodic. Then $(v,T,\rho) \in \E_{w,\tau}$ for all $\tau >0$ is called a weak $T^*$-periodic solution to \eqref{eq: primitive + EBM simplified}
  subject to \eqref{eq:bc} if
\begin{enumerate}
\item[(a)] $v$, $T$, $\rho$ are $T^*$-periodic.
 \item[(b)] For all $\tau>0$ and all $(\varphi, \tilde{\varphi}) \in \rC^1([0,\tau];\rX_w) \cap \rL^2(0,\tau;\rX_1)$ the following equality holds:
 \begin{equation}
     \begin{aligned}
           &\int_{0}^t \bigl[(v, \dt \varphi)_2-(\nabla v,\nabla \varphi)_2+(v \cdot \nablaH \varphi,v)_2+(w(v)\cdot \partial_z\varphi,v)+( \nablaH \theta,v)_2 \bigr]\d s \\ & \, + \int_0^t \int_\Omega \bigl[T \cdot \dt \tilde{\varphi} -\nabla T\cdot \nabla\tilde{\varphi}+u\cdot\nabla\tilde{\varphi} \cdot T\bigr]\d x  \d s
   \\ &  + \int_0^t \int_{G} \bigl[\rho \cdot \dt \psi  -\nablaH \rho \cdot \nablaH \psi + \vbar \cdot \nablaH \psi \cdot \rho   -R(x,\rho) \psi \bigr] \d \sigma\d s \\
           & = (v(t),\varphi(t))_2-(v(0),\varphi(0))_2+\int_{\Omega}\bigl[T(t)\tilde{\varphi}(t)-T(0)\tilde{\varphi}(0)\bigr] \d x +\int_{G}\bigl[\rho(t)\psi(t)-\rho(0)\psi(0)\bigr] \d \sigma
      \\
      & \quad -\int_0^t (f_1,\varphi)_2 \d s-\int_0^t\int_{\Omega} f_2 \cdot \tilde{\varphi} \d x\d s -\int_0^t\int_{G} f_3 \cdot \psi\d \sigma \d s,
     \end{aligned}
     \label{eq: (b) definition weak solution periodic}
    \end{equation}
    where $\psi = \tilde{\varphi}|_{\Gamma_u}$ and $w(v)$ and $\theta$ are defined, respectively, in \eqref{eq:w} and in \eqref{eq:press}.
\item[(c)] For all $\tau>0$ and all $t\in (0,\tau]$, and a.a. $s\in [0,t)$, the triple $(v,T,\rho)$ satisfies the {\em  energy inequality}
\begin{equation}
\begin{aligned}
    &\norm{v(t)}_{2}^2 +\norm{T(t)}_{2}^2 + \norm{\rho(t)}_{2}^2+ 2 \int_s^t (\norm{\nabla v(l)}_{2}^2 +\norm{\nabla T(l)}_{2}^2 + \norm{\nablaH\rho(l)}_{2}^2 +  \norm{\rho(l)}_{5}^5) \d l \\ &
\le \norm{v(s)}_{2}^2 +\norm{T(s)}_{2}^2 + \norm{\rho(s)}_{2}^2 + 2 \int_s^t \biggl[(f_1+ \nablaH \theta,v)_2+\int_\Omega f_2 \cdot T \d x+ \int_G (f_3 + Q(x)\beta(\rho)) \cdot \rho  \d \sigma \biggr]\d l.
\end{aligned}
\label{Xeqn11-11}
\end{equation}
\end{enumerate}
Furthermore, if a weak $T^*$-periodic solution $(v,T,\rho)$ to \eqref{eq: primitive + EBM simplified} subject to \eqref{eq:bc} satisfies $(v,T) \in \E_{1,\tau}$ for all $\tau>0$, then $(v,T,\rho)$ is
called a {\em strong $T^*$-periodic solution}.

Finally, if \((f_1, f_2, f_3) \in \rL^2(\rX_0)\) are time-independent, then  \((v, T, \rho) \in \rX_w\) is called a \emph{weak steady-state solution} to \eqref{eq: primitive + EBM simplified} subject to the boundary conditions \eqref{eq:bc}, if it satisfies \eqref{eq: (b) definition weak solution periodic} for all test functions \((\varphi, \tilde{\varphi}) \in \rX_1\). A weak steady-state solution \((v, T, \rho)\) is called \emph{strong} if \((v, T, \rho) \in \rX_1\).
}
  \end{definition}

We are now in the position to formulate our main theorem concerning the existence of strong $T^*$-periodic solutions in the presence of arbitrarily large forces.
\begin{theorem} \label{thm: main}  \mbox{} \\
 Let $T^*>0$ and $(f_1,f_2,f_3) \in \E_{0,\tau}$ be, for all $\tau>0$, $T^*$-periodic given external forces. Then the system \eqref{eq: primitive + EBM simplified} subject to the boundary conditions \eqref{eq:bc} has at least one strong $T^*$-periodic solution.
\label{Xenun1-1}
\end{theorem}

As a consequence, we conclude that if the external forces \((f_1, f_2, f_3)\) are time-independent, then there exists a steady-state solution to our energy balance model coupled with the primitive equations via a dynamic boundary condition.

\begin{corollary}    If \((f_1, f_2, f_3) \in \rL^2(\rX_0)\), then the system \eqref{eq: primitive + EBM simplified}, subject to the boundary conditions \eqref{eq:bc}, admits at least one strong steady-state solution.
\label{Xenun6-1}
\end{corollary}

\section{Global strong solutions}\label{Xsec3-3}\label{sec: global}
In this section we establish the existence of a unique global strong solution of \eqref{eq: primitive + EBM simplified} subject to the boundary conditions \eqref{eq:bc} and
arbitrarily large forces $f_1$, $f_2$ and $f_3$. The proof is a modification of a result obtained recently in \cite{bib0010}.

\begin{proposition}\label{prop: global with force}
Let $\tau>0$, $(v_0,T_0,\rho_0) \in \rX_w$ and $(f_1,f_2,f_3)\in \mathbb{E}_{0,\tau}$. Then the system \eqref{eq: primitive + EBM simplified} subject to \eqref{eq:bc} and \eqref{eq: initial data} admits a unique, global strong solution $(v,T,\rho)$ satisfying $(v,T,\rho)\in \E_{1,\tau}$.
\end{proposition}

\begin{proof}
    The existence of unique, local strong solutions follows directly from the considerations in \cite[Section~5]{bib0010}, since we are considering a slight modification of the model investigated in the reference. Hence, we can focus on showing that the local solution can be extended globally. Let us denote by $J_{\mathrm{max}}=[0,a_{\mathrm{max}}(T_0,v_0))$ the maximal time interval of existence, with $a_{\mathrm{max}}>0$. Thus, it suffices to show that
    \begin{equation*}
        \lim\limits_{\tau \to a_\mathrm{max}} \| ( v,T ,\rho) \|_{\E_{1,\tau}}  < \infty
    \end{equation*}
    in order to exclude any possible blow-up of the solution $(v,T,\rho)$. So, assume that $(v,T,\rho) \in \E_{1,a_{\mathrm{max}}}$ is the maximal, unique, local solution of \eqref{eq: primitive + EBM simplified} subject to \eqref{eq:bc}. Multiplying \eqref{eq: primitive + EBM simplified}$_{1,3,5}$ by $v$, $T$, and $\rho$ respectively, and integrating in time and space, we obtain, via H\"older's, Young's and Gronwall's inequalities, that there exists a continuous function $B_1$, depending on $\| (v_0,T_0,\rho_0) \|_{\rX_0}$, $\| (f_1,f_2,f_3) \|_{\E_{0,\tau}}$ and $t$, such that
    \begin{equation}\label{eq: energy global}
        \norm{v(t)}_{2}^2 +\norm{T(t)}_{2}^2 + \norm{\rho(t)}_{2}^2+ 2 \int_0^t \norm{\nabla v(s)}_{2}^2 +\norm{\nabla T(s)}_{2}^2 + \norm{\nablaH\rho(s)}_{2}^2 +  \norm{\rho(s)}_{5}^5 \d s \leq B_1(t),
    \end{equation}
    for all $t \in [0,a_\mathrm{max}]$. The energy inequality \eqref{eq: energy global} is analogous to the one obtained in \cite[Lemma~6.3]{bib0010}, in fact, for the transport term $\bar{v}\cdot \nabla_H \rho$, we have, thanks to an integration by parts,
    \[
     \int_G \bar{v}\cdot \nabla_H\rho \cdot  \rho=0.
    \]
    From \eqref{eq: energy global}, we conclude that
    \begin{equation*}
        \| \nablaH \theta \|_{\rL^2(0,a_\mathrm{max};\rL^2(\Omega))}  =   \|\nablaH \int_0^z T(\cdot,\xi) \d \xi \|_{ \rL^2((0,a_\mathrm{max}) \times \Omega)} \leq B_1(a_\mathrm{max}).
    \end{equation*}
    Then, following the arguments in \cite[Section~6]{bib0025}, there exists a continuous function $B_2$, depending on $\| (v_0,T_0) \|_{\rX_w^{v,T}}$ (where we set $\rX_w^{v,T}:=\rH^1_\per(\Omega;\R^2) \cap \rL^2_{\sigmabar}(\Omega;\R^2) \times \rH^1_\per(\Omega)$), $B_1$ and $t$, such that
    \begin{equation}\label{eq: L2H2 for v}
        \| \nabla v \|^2_2 + \int_0^t \| \Delta v(s) \|^2_2 \d s \leq B_2(t)
    \end{equation}
     for all $t \in [0,a_\mathrm{max}]$. Next, multiplying \eqref{eq: primitive + EBM simplified}$_{3,5}$ by $-\Delta T$ and $-\DeltaH \rho$, respectively, integrating in space, and adding the resulting equations, we obtain
       \begin{equation*}
        \begin{aligned}
            &\dt \bigl ( \frac{1}{2}\| \nabla T \|^2_2   + \| \nablaH \rho \|^2_2+ \frac{1}{5} \| \rho \|^5_5 \bigr )  + \| \Delta T \|^2_2 + \| \DeltaH \rho \|^2_2 \\ &\le
            \int_\Omega  \bigl (u \cdot \nabla T +f_2 \bigr) \cdot \Delta T + \int_G \left (\vbar \cdot \nablaH \rho +Q(x) \beta(\rho) +f_3 \right ) \cdot \DeltaH \rho  + \int_G (\dz T)|_{\Gamma_u}\cdot  \DeltaH \rho .
        \end{aligned}
    \end{equation*}
    Here, some terms on the right-hand side have already been estimated in \cite[Proposition~6.4]{bib0010}, hence we limit ourselves to estimate the terms including the external forces $f_2$ and $f_3$ and the different transport term on the boundary. By applying H\"older's and Young's inequalities, there exists a constant $C>0$ such that
    \begin{equation}\label{eq: l2H2 for T}
    \begin{aligned}
        &\bigl | \int_\Omega f_2 \cdot \Delta T \bigr  | \leq \| f_2 \|_2^2 + \eps \| \Delta T \|^2_2, \\ &\bigl | \int_G f_3 \cdot \DeltaH\rho  \bigr | \leq C \| f_3 \|^2_2 + \varepsilon\|\DeltaH \rho \|^2_2 ,\\
        &\bigl | \int_G \bar{v}\cdot \nablaH \rho \cdot \DeltaH \rho \bigr | \le \norm{\bar{v}}_{\infty}\norm{\nablaH \rho}_2\norm{\DeltaH \rho}_2\\ & \le \norm{v}_{\rH^2(\Omega)}\norm{\nablaH \rho}_2\norm{\DeltaH \rho}_2 \le C\norm{v}_{\rH^2(\Omega)}^2\norm{\nablaH \rho}_2^2 +\varepsilon \norm{\DeltaH \rho}_2^2.
        \end{aligned}
    \end{equation}
    Therefore, absorbing the highest order norms into the left-hand side, integrating in time and applying Gronwall's inequality, we deduce that there exists a continuous function $B_3$, depending on $\norm{(v_0,T_0,\rho_0)}_{\rX_w}$, $B_1$, and $B_2$ such that
    \begin{equation*}
    \| \nabla T \|^2_{2} + \| \nablaH \rho \|^2_{2} + \int_0^t  \| \Delta T(s) \|^2_2 + \| \DeltaH  \rho(s)  \|^2_{2} \d s \leq B_3(t)
\end{equation*}
for all $t \in [0,a_\mathrm{max}]$. Finally, by maximal $\rL^2$-regularity, it follows that
\begin{equation*}
  \begin{aligned}
     &\| (v,T) \|_{\E_{1,a_{\mathrm{max}}}}  \\ &\leq
    C \bigl ( \| \nablaH \theta - u \cdot \nabla v \|_{\rL^2((0,a_{\mathrm{max}}) \times \Omega)} + \| u \cdot \nabla T \|_{\rL^2((0,a_{\mathrm{max}}) \times \Omega)} +\|  R(x,\rho)- \vbar \cdot \nablaH \rho \|_{\rL^2((0,a_{\mathrm{max}}) \times G)}  \\
    &\quad + \| f_1 \|_{\rL^2((0,a_{\mathrm{max}}) \times \Omega)} + \| f_2 \|_{\rL^2((0,a_{\mathrm{max}}) \times \Omega)} + \| f_3 \|_{\rL^2((0,a_{\mathrm{max}}) \times G)} \bigr ).
  \end{aligned}
\end{equation*}
We see that the bounds derived in \eqref{eq: L2H2 for v} and \eqref{eq: l2H2 for T} are sufficient to control the solution norm of $(v,T,\rho)$, provided that all external forces possess at least $\rL^2_t\,-\,\rL^2_x$-regularity in their respective domains.
\end{proof}

\section{Weak time periodic solutions}\label{Xsec4-4} \label{sec:WS}
In this section, we prove an existence result for weak $T^*$-periodic solutions, for arbitrary large data. The existence of weak time-periodic solutions to a diffusive energy balance model was discussed for example in the works of Diaz and Tello \cite{bib0013}, and by Badii and Diaz \cite{bib0003}. However, they do not couple the EBM with fluid equations.

The main result of this section is the following.
\begin{proposition}\label{prop: weak time periodic}
Let $T^*>0$, and let $(f_1,f_2,f_3)$ be as in \autoref{thm: main}. Then, there exists at least one weak $T^*$-periodic solution $(v,T,\rho)$ to the system \eqref{eq: primitive + EBM simplified} in the sense of \autoref{def: weak sol}.
\end{proposition}

\begin{proof}
    Let $\{(\psi_k, \tilde{\psi}_k)\}_{k\in \mathbb{N}}\subset \rX_1$ be an orthonormal basis of $\rX_{0}$ that is dense both in $\rX_w$ and in $\rX_1$. For instance, we may choose $\psi_k$
 to be eigenfunctions of the \emph{Neumann-Laplacian} and $\tilde{\psi}_k$ to be eigenfunctions of the Laplacian subject to dynamic boundary conditions, see \cite[Section 5]{bib0010}.
    \par

    We construct approximate solutions by setting
\[
v^m(t,x):=\sum_{k=1}^m c_{k}^m(t)\psi_k(x),\quad T^m(t,x):=\sum_{k=1}^m c_{k}^m(t)\tilde{\psi}_k(x),
\]
and, for the surface temperature,
\[
\rho^m(t,x):=T^m|_{\Gamma_u}(t,x_\H)=\sum_{k=1}^m c_{k}^m(t)\tilde{\psi}_{k,\Gamma}(x),
\]
where we define $\tilde{\psi}_{k,\Gamma}:=\tilde{\psi}_k|_{\Gamma_u}$.
The coefficients $\{c_{k}^m(t)\}_{k=1}^m$ satisfy the ordinary differential equations
    \[
    \begin{aligned}
    &(v^{m}_t, \psi_k)+(\nabla v^m, \nabla\psi_k)-(v^m \cdot \nablaH \psi_k+ w^m(v^m)\partial_z \psi_k, v^m)-(\nablaH \int_0^z T^m(\cdot,\xi)\,\d \xi, \psi_k)+\\
    &+\int_{\Omega}T^{m}_t\tilde{\psi}_k + \int_{\Omega}\nabla T^m\cdot\nabla \tilde{\psi}_k-\int_{\Omega}u^m \cdot \nabla\tilde{\psi}_k\cdot T^m+ \int_{G}\rho^{m}_t \tilde{\psi}_{k,\Gamma}-\int_G \bar{v} \cdot \nablaH \tilde{\psi}_{k,\Gamma}\cdot \rho^m+\\ &+\int_G\nablaH \rho^m\cdot\nablaH \tilde{\psi}_{k,\Gamma}-\int_G R(x,\rho^m) \tilde{\psi}_{k,\Gamma}\,dx_\H=(f_1,\psi_k)+\int_{\Omega}f_2\tilde{\psi}_k+ \int_G f_3 \tilde{\psi}_{k,\Gamma}, \quad k=1,\dots, m.
    \end{aligned}
    \]
    Multiplying the previous equation by $c_{k}^m(t)$ and summing over $k$, we get
    \begin{equation} \label{eq:sys}
    \begin{aligned}
    &(v^{m}_t, v^m)+(\nabla v^m, \nabla v^m)-(\nablaH \int_0^z T^m(\cdot,\xi) \d\xi, v^m)+\int_{\Omega} T^{m}_tT^m + \int_{\Omega}\nabla T^m\cdot \nabla T^m+ \\
    &+\int_G\rho^{m}_t \rho^m+\int_G\nablaH \rho^m\cdot\nablaH \rho^m-\int_G R(x,\rho^m) \rho^m=(f_1,v^m)+\int_{\Omega} f_2 T^m + \int_G f_3 \rho^m,
    \end{aligned}
    \end{equation}
    thanks to
    \[
    (v^m \cdot \nablaH v^m+ w^m(v^m)\partial_z v^m, v^m)=\int_{G}\bar{v} \cdot \nablaH \rho^m\cdot \rho^m =\int_{\Omega} u^m \cdot \nabla T^m\cdot T^m=0.
    \]
    Using the estimate
    \[
    -(\nablaH \int_0^z T^m(\cdot,\xi)\d\xi, v^m) \ge  -\frac{1}{2}\norm{\nabla T^m}_{2}^2- \frac{1}{2}\norm{v^m}_{2}^2,
    \]
    we deduce, via H\"older's and Young's inequalities, that
    \[
    \begin{aligned}
    &\frac{d}{dt}(\norm{v^m}_{2}^2+\norm{T^m}_{2}^2+\norm{\rho^m}_{2}^2)+2(\norm{\nabla v^m}_{2}^2+\norm{\nabla T^m}_{2}^2+\norm{\nablaH \rho^m}_{2}^2)\\
    &-\norm{\nabla T^m}_{2}^2 - \norm{v^m}_{2}^2+2\norm{\rho^m}_{5}^5- 2\int_G Q(x)\beta(\rho^m)\rho^m\\
    &\le \norm{f_1}_{2}^2+\norm{v^m}_{2}^2+\norm{f_2}_{2}^2+\norm{T^m}_{2}^2+\norm{f_3}_{2}^2+\norm{\rho^m}_{2}^2.
    \end{aligned}
    \]
    Setting $y^m(t):=2\norm{v^m(t)}_{2}^2+\norm{T^m(t)}_{2}^2+2\norm{\rho^m(t)}_{2}^2$, we deduce that
    \[
     \frac{d}{dt}y^m \le y^m+ \Phi, \ \text{ with } \
    \Phi:=C^2+\norm{f_1}_{2}^2+\norm{f_2}_{2}^2+\norm{f_3}_{2}^2.
    \]
    Gronwall's lemma implies
    \begin{equation} \label{eq:Gr}
    y^m(t)\le \mathrm{e}^t \bigl (y^m(0)+\int_0^t\Phi(s) \d s \bigr).
    \end{equation}
    We hence obtain $c^m(t)=\norm{v^m(t)}_2=\norm{T^m(t)}_2=\norm{\rho^m(t)}_{2}$. And choosing $R>0$ sufficiently large,
    we obtain
    \[
    y^m(T^*)=5\abs{c^m(T^*)}^2 \le e^{T^*}\bigl (5\abs{c^m(0)}^2+\int_{0}^{T^*}\Phi(s)\d s\bigr )\le R^2
    \]
    provided that $y^m(0)=5 \cdot \abs{c^m(0)}^2\le R^2$. Hence, the map
    \[
    S:c^m(0)\in \mathbb{R}^m \mapsto c^m(T^*) \in \mathbb{R}^m
    \]
    is a self-map on the ball $\rB^{m}_{\frac{R}{\sqrt{5}}}:=\{x \in \mathbb{R}^m \ : \ \abs{x-0} \le \frac{R}{\sqrt{5}}\}$ and it is continuous. Brouwer's fixed point theorem yields the existence of $c^m(0) \in \rB_{\frac{R}{\sqrt{5}}}^m$ such that $c^m(0)=c^m(T^*)$. Hence, there exist $v^m(0)$, $T^m(0)$, $\rho^m(0)$ such that
    \[
    v^m(0)=v^m(T^*), \quad T^m(0)=T^m(T^*), \quad \rho^m(0)=\rho^m(T^*).
    \]
    Thus, $(v^m(t),T^m(t),\rho^m(t))$ can be extended to $[0,\infty)$ as periodic functions with period $T^*$. Following the arguments in \cite{bib0016,bib0038}, we get that
    \[
    (v^m,\psi), \quad  \int_{\Omega}T^m\tilde{\psi}, \quad \int_G\rho^m\tilde{\psi}_{\Gamma},
    \]
    for $(\psi,\tilde{\psi},\tilde{\psi}_{\Gamma})\in \rX_0$, are uniformly bounded and uniformly continuous.
    Hence, they admit converging subsequences to $g(t)$, $h(t)$, $k(t)$ in $\rC([0,\tau))$, and  converging weakly to a limit $(v,T,\rho)$ in $\E_{0,\tau}$. On the other hand, using
\[
    \frac{d}{dt} {y}^m +2(\norm{\nabla v^m}_{2}^2+\norm{\nabla T^m}_{2}^2+\norm{\nablaH \rho^m}_{2}^2)\le y^m+ \Phi,
    \]
together with \eqref{eq:Gr} and the fact that $y^m(0)\le R^2$, these subsequences have a weak limit $(\hat{v}, \hat{T}, \hat{\rho})$ in $\rL^2(0,\tau;\rX_1)$. Friedrichs lemma then ensures that
$(v^{m}, T^{m}, \rho^{m})$ is a Cauchy sequence in $\E_{0,\tau}$ and thus it converges strongly to $(\hat{v}, \hat{T}, \hat{\rho})$ in that space. It is now standard to show that $(v,T,\rho)$ satisfies the weak formulation in the sense of \autoref{def: weak sol}. Finally, concerning the energy inequality, noticing that
\[
    \begin{aligned}
        & (v^m, T^m, \rho^m) \to (v,T,\rho) \,\, \text{weakly in} \,\, \rL^2(0,\tau;\rX_1), \\
        & v^m(t) \to v(t), \,\, T^m(t)\to T(t), \,\, \rho^m(t) \to \rho(t),\,\, \text{weakly in}\,\, \rX_0, \,\, \text{ for all} \,\, t\in[0,\tau), \\
        & (v^m, T^m, \rho^m) \to (v,T,\rho) \,\, \text{strongly in} \,\, \E_{0,\tau},
    \end{aligned}
    \]
we can pass to the limit $m\to+\infty$ in \eqref{eq:sys} (after integrating in time) to obtain the desired energy inequality. This completes the proof.
\end{proof}

\section{Weak-strong uniqueness}\label{Xsec5-5} \label{sec:w-s-u}
\noindent
In this section, we prove a weak-strong uniqueness result for our model.
\begin{proof}[Proof of Theorem \ref{thm: main}]
Let $(v_1,T_1,\rho_1)$ be a weak $T^*$-periodic solution to \eqref{eq: primitive + EBM simplified} subject to \eqref{eq:bc}. Then, by \autoref{def: weak sol} there exists $t_0>0$ such that $(v_1,T_1)(t_0) \in\rX_w$. Let $(v_2,T_2,\rho_2)$ be the strong solution to the system \eqref{eq: primitive + EBM} given by \autoref{prop: global with force} with initial condition $(v_1,T_1)(t_0)$. Our goal is to prove that
\begin{equation}
(v_1,T_1,\rho_1) = (v_2,T_2,\rho_2) \quad \text{ a.\ e. in }  (t_0,t) \times  \Omega \times \Omega \times G, \quad \text{ for all } t> t_0.
\label{eq: thesis weak strong uniqueness}
\end{equation}
By estimate \eqref{Xeqn11-11}, it holds that
\begin{equation}
\begin{split}
    &\norm{v_1}_2^2 + \norm{T_1}_2^2 + \norm{\rho_1}_2^2 + 2 \int_{t_0}^t \left \lbrace\norm{\nabla v_1(s)}_2^2 + \norm{\nabla T_1(s)}_2^2 + \norm{\nablaH \rho_1(s)}_2^2 + \norm{\rho_1(s)}_5^5 \right \rbrace \d s \le \\
      &\norm{v_1(t_0)}_2^2 + \norm{T_1(t_0)}_2^2 + \norm{\rho_1(t_0)}_2^2 + 2 \int_{t_0}^t \bigl \{ (f_1 + \nablaH \theta_1,v_1)_2+ \int_{\Omega} f_2 \cdot T_1 + \int_{G}\bigl (f_3 + Q(x) \cdot  \beta(\rho_1)\bigr )\cdot \rho_1 \bigr \} \d s
\end{split}
    \label{eq: 5.1 bis}
\end{equation}
for each $t \geq t_0$, where $\theta_1 (\cdot, z) = \int_0^z T_1(\cdot, \xi) \d \xi$. Further, via \eqref{eq: (b) definition weak solution periodic} we infer that for any $\mathcal{T}>0$ and $t \in (t_0, \mathcal{T}]$ the following holds:
\begin{equation}
\begin{split}
    &\int_{t_0}^t \left \lbrace \left(v_1, \partial_t \phi\right)_2 - \left(\nabla v_1, \nabla \phi\right)_2 - \left(v_1 \cdot \nablaH v_1 + w(v_1) \partial_z v_1, \phi\right)_2  + (\nablaH \theta_1 , v_1)_2 \right \rbrace \d s  \\
    &\quad+\int_{t_0}^t  \left \lbrace  \int_{\Omega}T_1 \cdot  \partial_t \tilde{\phi} -  \nabla T_1 \cdot  \nabla \tilde{\phi} +  u_1 \cdot \nabla \tilde{\phi} \cdot  T_1 \right \rbrace \d s\\
    &\quad+ \int_{t_0}^t \left \lbrace \int_{G}\phi \cdot \partial_t \psi - \nablaH \rho_1 \cdot  \nablaH \psi+ \bar{v}_1 \cdot \nablaH \psi  \cdot \rho_1 - R(x ,\rho_1 ) \cdot  \psi \right \rbrace \d s \\
    &= (v_1(t), \phi(t))_2 - (v_1(t_0), \phi(t_0))_2 +\int_\Omega T_1(t) \cdot \tilde{\phi}(t) - \int_\Omega T_1(t_0)\cdot  \tilde{\phi}(t_0) + \int_G \rho_1(t)\cdot \psi(t)- \int_G\rho_1(t_0)\cdot \psi(t_0)\\
    &\quad- \int_{t_0}^t \left \lbrace (f_1, \phi)_2 + \int_{\Omega}f_2 \cdot  \tilde{\phi} + \int_{G}f_3 \cdot  \psi  \right\rbrace \d s,
\end{split}
\label{eq: 5.2 bis}
\end{equation}
for all $( \phi, \tilde{\phi}) \in \rC^1([t_0,\mathcal{T}]; \rX_w) \cap \rL^2([t_0,\mathcal{T}]; \rX_1)$, where $\psi = \tilde{\varphi}|_{\Gamma_u}$. Moreover, we have used the identity
\begin{equation}
(v_1 \cdot \nablaH \phi + w(v_1) \partial_z \phi, v_1)_2 = - (v_1 \cdot \nablaH v_1 + w(v_1) \partial_z v_1, \phi)_2,
    \label{eq: 5.3 bis}
\end{equation}
for $v_1, \phi$ as above. Considering now the strong solution $(v_2, T_2,\rho_2)$ with initial condition $(v_2, T_2,\rho_2)(t_0) = (v_1, T_1,\rho_1)(t_0) $, we can deduce an identity analogous to \eqref{eq: 5.2 bis}, i.\ e.,
\begin{equation}
\begin{aligned}
    &\int_{t_0}^t \left \lbrace \left(v_2, \partial_t \phi\right)_2 - (\nabla v_2, \nabla \phi)_2 + (v_2 \cdot \nablaH \phi + w(v_2) \partial_z \phi, v_2)_2  + (\nablaH \theta_2 , v_2)_2 \right \rbrace \d s  \\
    &\quad +\int_{t_0}^t  \left \lbrace \int_\Omega T_2 \cdot \partial_t \tilde{\phi} -   \nabla T_2 \cdot \nabla \tilde{\phi} +  u_2 \cdot \nabla \tilde{\phi}\cdot  T_2 \right \rbrace\\
    &\quad+ \int_{t_0}^t \left \lbrace \int_G\phi\cdot \partial_t \psi - \nablaH \rho_2 \cdot \nablaH \psi+  \bar{v}_2 \cdot \nablaH \psi \cdot \rho_2 - R(x ,\rho_2 ) \cdot \psi \right \rbrace \d s \\
    &= (v_2(t), \phi(t))_2 - (v_2(t_0), \phi(t_0))_2 +\int_\Omega T_2(t) \cdot \tilde{\phi}(t) - \int_\Omega T_2(t_0) \cdot \tilde{\phi}(t_0) + \int_G \rho_2(t)\cdot \psi(t) -\int_G \rho_2(t_0) \cdot \psi(t_0)\\
    &\quad - \int_{t_0}^t  \lbrace  (f_1, \phi)_2 + \int_\Omega f_2 \cdot \tilde{\phi} + \int_G f_3 \cdot \psi \rbrace \d s,
\end{aligned}
    \label{eq: 5.4 bis}
\end{equation}
for all $t \in (t_0, \mathcal{T}]$, and test functions $(\phi, \tilde{\phi})$. Here, $\theta_2 (\cdot, z) = \int_0^z T_2(\cdot, \xi) \d \xi$. Further, the following energy equality holds true
\begin{equation}
\begin{split}
    &\norm{v_2(t)}_2^2 + \norm{T_2(t)}_2^2 + \norm{\rho_2(t)}_2^2 + 2 \int_{t_0}^t \left \lbrace\norm{\nabla v_2(s)}_2^2 + \norm{\nabla T_2(s)}_2^2 + \norm{\nablaH \rho_2(s)}_2^2 + \norm{\rho_2(s)}_5^5 \right \rbrace \d s=\\
     &\norm{v_2(t_0)}_2^2 + \norm{T_2(t_0)}_2^2 + \norm{\rho_2(t_0)}_2^2 + 2 \int_{t_0}^t \left \lbrace (f_1 + \nablaH \theta_2,v_2)_2+ \int_\Omega f_2\cdot T_2 + \int_G (f_3 + Q(x) \beta(\rho_2)) \cdot \psi \right \rbrace \d s,
\end{split}
    \label{eq: 5.5 bis}
\end{equation}
for all $t \geq t_0.$

For $i =1,2,$ we consider the mollifications
\[
v_{i,h}(t) := \int_0^{\mathcal{T}} j_h(t-s) \cdot  v_i (s) \d s, \quad T_{i,h}:= \int_0^{\mathcal{T}} j_h(t-s) \cdot T_i(s) \d s, \quad \rho_{i,h}(t) := \int_0^{\mathcal{T}} j_h(t-s) \cdot  \rho_i (s) \d s,
\]
where $j_h \in \rC^\infty_c (-h,h)$, for $0 < h <\mathcal{T}$, is an even, positive mollifier with $\norm{j_h}_1 = 1.$ Further, by the regularity of the mollifications, it holds that
\begin{equation}
\begin{aligned}
    \int_{t_0}^t (v_1, \partial_t v_{2,h})_2 \d s & = -\int_{t_0}^t (v_{1,h},v_2)_2  \d s + (v_{2,h}(t),v_1(t))_2 -(v_{2,h}(t_0),v_1(t_0))_2 \\
    \int_{t_0}^t \int_{\Omega} T_1 \cdot  \partial_t T_{2,h}  \d s & =- \int_{t_0}^t \int_{\Omega}T_{1,h} \cdot T_2  \d s +  \int_{\Omega} T_{2,h}(t) \cdot T_1(t) -\int_{\Omega}T_{2,h}(t_0) \cdot T_1(t_0) \\
    \int_{t_0}^t \int_{G} \rho_1 \cdot  \partial_t \rho_{2,h} \d s & =- \int_{t_0}^t \int_G \rho_{1,h} \cdot \rho_2 \d s + \int_G\rho_{2,h}(t) \cdot \rho_1(t) -\int_G \rho_{2,h}(t_0) \cdot \rho_1(t_0)
\end{aligned}
    \label{eq: 5.7 bis}
\end{equation}
and
\begin{equation}
    \begin{aligned}
&\lim_{h \to 0} (v_2(t), v_{1,h}(t))_2 \,&=\,& \lim_{h \to 0} (v_{2,h}(t), v_1(t))_2 \,&=\,& (v_2(t), v_1(t))_2 \\
&\lim_{h \to 0} \int_{\Omega} T_2(t) \cdot T_{1,h}(t) \,&=\,& \lim_{h \to 0} \int_{\Omega} T_{2,h}(t) \cdot T_1(t) \,&=\,& \int_{\Omega} T_2(t) \cdot T_1(t) \\
&\lim_{h \to 0} \int_{G} \rho_2(t) \cdot \rho_{1,h}(t) \,&=\,& \lim_{h \to 0} \int_{G} \rho_{2,h}(t) \cdot \rho_1(t) \,&=\,& \int_{G} \rho_2(t) \cdot \rho_1(t)
\end{aligned}
    \label{eq: 5.8 bis}
\end{equation}
for all $t \geq t_0$.

Considering $(\phi, \tilde{\phi}) = (v_{2,h}, T_{2,h})$ in $\eqref{eq: 5.2 bis}$, we deduce
\begin{equation}
\begin{split}
    &\int_{t_0}^t \left \lbrace (v_1, \partial_t v_{2,h})_2 - (\nabla v_1, \nabla v_{2,h})_2 - (v_1 \cdot \nablaH v_1 + w(v_1) \partial_z v_1, v_{2,h})_2  + (\nablaH \theta_1 , v_1)_2 \right \rbrace \d s  \\
    &\quad+\int_{t_0}^t  \left \lbrace \int_{\Omega}T_1 \cdot  \partial_t T_{2,h} -  \nabla T_1 \cdot \nabla T_{2,h} +  u_1 \cdot \nabla T_{2,h} \cdot  T_1 \right \rbrace \d s\\
    &\quad + \int_{t_0}^t \left \lbrace \int_{G} v_{2,h} \cdot  \partial_t \psi_{2,h} - \nablaH \rho_1 \cdot \nablaH \psi_{2,h} + \bar{v} \cdot \nablaH \psi_{2,h} \cdot \rho_1 - R(x ,\rho_1 ) \cdot  \psi_{2,h}\right \rbrace \d s \\
    &= (v_1(t), v_{2,h}(t))_2 - (v_1(t_0), v_{2,h}(t_0))_2 + \int_{\Omega}T_1(t) \cdot  T_{2,h}(t) - \int_{\Omega} T_1(t_0) \cdot  T_{2,h}(t_0) \\
    &\quad + \int_G \rho_1(t) \cdot  \psi_{2,h}(t) - \int_{G} \rho_1(t_0)\cdot \psi_{2,h}(t_0)- \int_{t_0}^t \left \lbrace  (f_1, v_{2,h})_2 + \int_{\Omega}f_2 \cdot T_{2,h} + \int_{G}f_3 \cdot  \psi_{2,h} \right \rbrace \d s,
\end{split}
    \label{eq: 5.2 ter}
\end{equation}
where $\psi_{2,h} = T_{2,h}|_{\Gamma_u}$. We then consider $(\phi, \tilde{\phi}) = (v_{1,h}, T_{1,h})$ in \eqref{eq: 5.4 bis} and use \eqref{eq: 5.3 bis} to obtain
\begin{equation}
\begin{aligned}
    &\int_{t_0}^t \left \lbrace \left(v_2, \partial_t v_{1,h}\right)_2 - \left(\nabla v_2, \nabla v_{1,h}\right)_2 - \left(v_2 \cdot \nablaH v_2 + w(v_2) \partial_z v_2, v_{1,h}\right)_2  + \left(\nablaH \theta_2 , v_2\right)_2 \right \rbrace \d s  \\
    &\quad +\int_{t_0}^t  \left \lbrace \int_{\Omega} T_2 \cdot  \partial_t T_{1,h} - \nabla T_2 \cdot  \nabla T_{1,h} +  u_2 \cdot \nabla T_{1,h} \cdot  T_2 \right \rbrace\\
    &\quad + \int_{t_0}^t \left \lbrace \int_{G}\rho_2 \cdot  \partial_t \psi_{1,h} - \nablaH \rho_2 \cdot  \nablaH \psi_{1,h} + \bar{v}_2 \cdot \nablaH \psi_{1,h} \cdot  \rho_2 - R(x ,\rho_2) \cdot  \psi_{1,h} \right \rbrace \d s \\
    &= \left(v_2(t), v_{1,h}(t)\right)_2 - \left(v_1(t_0), v_{1,h}(t_0)\right)_2 + \int_{\Omega} T_2(t) \cdot  T_{1,h}(t) - \int_{\Omega} T_2(t_0) \cdot  T_{1,h}(t_0)\\
    &\quad + \int_G \rho_2(t) \cdot  \psi_{1,h}(t) - \int_G \rho_2(t_0) \cdot  \psi_{1,h}(t_0)
    - \int_{t_0}^t  \lbrace  (f_1, v_{1,h})_2 + (f_2, T_{1,h})_2 + (f_3, \psi_{1,h})_2 \rbrace \d s,
\end{aligned}
    \label{eq: 5.4 ter}
\end{equation}
with $\psi_{2,h} = T_{1,h}|_{\Gamma_u}.$ Denote from now on the difference between the weak and the strong solution by
\begin{equation*}
        (\sigma_v, \sigma_T, \sigma_\rho) = (v_1, T_1, \rho_1) - (v_2, T_2, \rho_2),
\end{equation*}
and similarly
\[
\sigma_u = u_1 - u_2, \quad \sigma_{\bar{v}} = \bar{v}_1 - \bar{v}_2.
\]
By summing \eqref{eq: 5.2 ter} and \eqref{eq: 5.4 ter}, and considering the limit in $h$, classical results on mollifications, \eqref{eq: 5.7 bis}, \eqref{eq: 5.8 bis} along with
\[
( \sigma_v \cdot \nablaH v_2 + w(\sigma_v) \cdot \partial_z v_2, v_2 )_2 = 0,
\]
we obtain
\begin{equation}
\begin{aligned}
    & -\int_{t_0}^t \left \lbrace 2 \cdot  ( \nabla v_1, \nabla v_2)_2 + ( \sigma_v \nablaH \sigma_v + w(\sigma_v)\cdot \partial_z \sigma_v,v_2)_2 \right \rbrace \d s
- \int_{t_0}^t \left \lbrace  \int_{\Omega} 2 \cdot\nabla T_1 \cdot  \nabla T_2 + \sigma_u \cdot \nabla \sigma_T \cdot  T_2 \right \rbrace \d s \\
& \quad -\int_{t_0}^t \left \lbrace \int_G 2 \cdot   \nablaH \rho_1 \cdot  \nablaH \rho_2 +\sigma_{\bar{v}} \cdot \nablaH \sigma_\rho \cdot \rho_2 + R(x, \rho_1) \cdot \rho_2  + R(x, \rho_2) \cdot \rho_1  \right \rbrace  \d s \\
= & (v_1(t), v_2(t))_2 - \norm{v_1(t_0)}_2^2 - \int_{t_0}^t \left \lbrace ( \nablaH \theta_1, v_1)_2 + (\nablaH \theta_2,v_2)_2 + (f_1,v_1 + v_2)_2 \right \rbrace \d s \\
& \quad + (T_1(t), T_2(t))_2 - \norm{T_1(t_0)}_2^2  - \int_{t_0}^t \int_{\Omega}f_2 \cdot ( T_1 + T_2) \d s \\
& \quad + (\rho_1(t), \rho_2(t))_2 - \norm{\rho_1(t_0)}_2^2  - \int_{t_0}^t \int_G f_3 \cdot ( \rho_1 + \rho_2) \d s.
\end{aligned}
    \label{eq: 5.10 bis}
\end{equation}
Adding twice \eqref{eq: 5.10 bis} to \eqref{eq: 5.1 bis} and \eqref{eq: 5.5 bis}, we arrive at
\begin{equation}
\begin{split}
    &\norm{\sigma_v(t)}_2^2 + \norm{\sigma_T(t)}_2^2 + \norm{\sigma_\rho (t)}_2^2 + 2 \cdot \int_{t_0}^t \left \lbrace\norm{\nabla \sigma_v}_2^2 + \norm{\nabla \sigma_T}_2^2 + \norm{\nablaH \sigma_\rho}_2^2\right \rbrace \d s \\
    & \leq   2 \cdot \int_{t_0}^t \left \lbrace (\sigma_v \cdot \nablaH \sigma_v + w(\sigma_v) \cdot \partial_z \sigma_v, v_2)_2 + \int_{\Omega}\sigma_u \cdot \nabla \sigma_T \cdot  T_2 \right \rbrace \d s \\
    & \quad + 2 \cdot  \int_{t_0}^t \left \lbrace \int_{G}  \sigma_{\bar{v}}\cdot \nablaH \sigma_{\rho} \cdot  \rho_2 +R(x, \rho_1) \cdot \rho_2 + R(x, \rho_2) \cdot  \rho_1   \right \rbrace \d s.
\end{split}
    \label{eq: Inequality for grownall lemma weak strong uniqueness}
\end{equation}
Using \autoref{lem: estimate terms for weak strong uniqueness}, we bound the terms on the right hand side of \eqref{eq: Inequality for grownall lemma weak strong uniqueness}, obtaining
\begin{equation*}
    \begin{split}
    &\norm{\sigma_v(t)}_2^2 + \norm{\sigma_T(t)}_2^2 + \norm{\sigma_\rho (t)}_2^2 + 2 \cdot \int_{t_0}^t \left \lbrace\norm{\nabla \sigma_v}_2^2 + \norm{\nabla \sigma_T}_2^2 + \norm{\nablaH \sigma_\rho}_2^2\right \rbrace \d s \\
         &\leq  \int_{t_0}^t  \left \lbrace C(\varepsilon) g(s) \norm{\sigma_v}_2^2 + \varepsilon \left( \norm{\sigma_v}_{\rH^1}^2+ \norm{\sigma_T}_{\rH^1}^2 + \norm{\sigma_\rho}_{\rH^1}^2\right) + C\cdot \left( \norm{\rho_1}_5^5 + \norm{\rho_2}^5_5 \right) \right \rbrace \d s,
    \end{split}
\end{equation*}
where $\varepsilon>0$ is arbitrarly small, $C(\varepsilon)$ is positive and bounded, $C>0$ is a constant independent of $\varepsilon$, and $g \in \rL^1(t, t_0)$ for all $t > t_0$. We then conclude applying Gronwall's inequality, which yields \eqref{eq: thesis weak strong uniqueness}.
\end{proof}
It remains to obtain the bounds for the terms on the right-hand side of \eqref{eq: Inequality for grownall lemma weak strong uniqueness}, which are key to concluding the Gronwall argument for uniqueness.

\begin{lemma}\label{lem: estimate terms for weak strong uniqueness}
Let $(v_1, T_1, \rho_1)$ and $(v_2,T_2,\rho_2)$ be, respectively, a weak and a strong solution of \eqref{eq: primitive + EBM simplified} such that $(v_i, T_i)(t_0) = (v_1,T_1)(t_0) \in \rX_0$ for $i = 1,2.$ Then, for any $t>t_0$, it holds true that
\begin{align}
    \Bigg|\int_{t_0}^t \int_{\Omega} \sigma_v \cdot \nablaH T_2 \cdot \sigma_T  \d s  \Bigg| &\leq C(\eps) \int_{t_0}^t g_1(s) \norm{\sigma_v}_2^2  + \eps \int_{t_0}^t \left \lbrace \norm{\sigma_v}_{\rH^1}^2 + \norm{\sigma_T}_{\rH^1}^2 \right \rbrace \d s \tag{i},\\
    \Bigg|\int_{t_0}^t \int_{\Omega} w(\sigma_v) \cdot \partial_z T_2 \cdot \sigma_T  \d s  \Bigg| &\leq C(\eps) \int_{t_0}^t g_2(s) \norm{\sigma_T}_2^2  + \eps \int_{t_0}^t \left \lbrace \norm{w(\sigma_v)}_{\rH^1}^2 + \norm{\sigma_T}_{\rH^1}^2 \right \rbrace \d s \tag{ii},\\
    \Bigg|\int_{t_0}^t \int_{G} \sigma_{\vbar} \cdot \nablaH \rho_2 \cdot \sigma_\rho \d s  \Bigg| &\leq C(\eps) \int_{t_0}^t g_3(s) \norm{\sigma_v}_2^2  + \eps \int_{t_0}^t \left \lbrace \norm{\sigma_v}_{\rH^1}^2 + \norm{\sigma_\rho}_{\rH^1}^2 \right \rbrace \d s \tag{iii},\\
    \Bigg|\int_{t_0}^t \int_{G} R(x, \rho_1)\cdot \rho_2 + R(x, \rho_2) \cdot \rho_1  \d s  \Bigg| &\leq C \left( \norm{\rho_1}^5_{\rL^5((t_0,t)\times G)}+ \norm{\rho_2}^5_{\rL^5((t_0,t)\times G)} \right), \tag{iv}
\end{align}
where $g_i$ is non negative, $g_i \in \rL^1(t_0,t)$ for $i = 1,2,3$, $C(\eps)$ is positive, and $C>0$ is independent of $\eps$.
\end{lemma}
\begin{proof}
    (i) Applying the triangular inequality and H\"older's inequality, we have
    \[     \Bigg|\int_{\Omega}\sigma_v \cdot \nablaH T_2 \cdot \sigma_T  \Bigg| \leq \norm{\sigma_v}_3 \cdot  \norm{\nablaH T_2}_2 \cdot \norm{\sigma_T}_6 \leq C \cdot \norm{\sigma_v}_3 \cdot  \norm{\nablaH T_2}_2 \cdot \norm{\sigma_T}_{\rH^1},    \]
where the last inequality follows from the embedding $\rH^1(\Omega) \hookrightarrow \rL^6(\Omega).$ Then, by interpolation inequality and Young's inequality, we obtain, for each $\varepsilon >0$,
\begin{equation*}
\begin{split}
\norm{\sigma_v}_3 \cdot  \norm{\nablaH T_2}_2 \cdot \norm{\sigma_T}_{\rH^1} &\leq \norm{\sigma_v}_2^{{1}/{2}} \cdot \norm{\sigma_v}_{\rH^1}^{{1}/{2}} \cdot \norm{T_2}_{\rH^1} \cdot \norm{\sigma_T}_{\rH^1} \\
& \leq  C(\varepsilon) \cdot \norm{\sigma_v}_2^2 \cdot \norm{T_2}_{\rH^1}^4  + \varepsilon \cdot  \left(\norm{\sigma_T}_{\rH^1}^2  + \norm{\sigma_v}_{\rH^1}^2\right).
\end{split}
\end{equation*}
The desired inequality then follows by integrating in time and setting $g_1(s):= \norm{T_2(s)}_{\rH^1}^4$, which is integrable since $(v_2,T_2)$ is a strong solution.

(ii)-(iii) Inequality (ii) is obtained by an application of Young's inequality after the use of (93) from \cite{bib0007}. Furthermore, inequality (iii) is derived with the same strategy as in (i).

(iv) Since there is a symmetry in the left-hand side of (iv), we only show how to bound $\abs{\int_{t_0}^t \int_GR(x_\H, \rho_1) \cdot \rho_2 \d s}$. By the definition of the radiation term $R$, we have
\[
\int_G \abs{R(x,\rho_1) \cdot \rho_2} \leq \int_G  \abs{Q(x) \cdot \beta(\rho_1) \cdot \rho_2}  + \int_{G} \abs{\rho_1^4 \cdot \rho_2}  \leq C \left( \norm{\rho_2}_1 + \norm{\rho_1}_5^5 + \norm{\rho_2}_5^5 \right),
\]
where we have used that $Q,\beta$ are bounded, and Young's inequality. The thesis follows by integrating in time and using the embedding $\rL^5(G) \hookrightarrow \rL^1(G).$
\end{proof}

\section*{Acknowledgements}
{Gianmarco Del Sarto, Matthias Hieber and Tarek Z\"{o}chling acknowledge the support from the DFG Research Unit  FOR~5528. Filippo Palma acknowledges the support from the GNFM research group of the \emph{Istituto Nazionale di Alta Matematica}.}

\end{document}